\documentclass[12pt]{amsart}

\usepackage{amssymb,color}
\usepackage{amsfonts}
\usepackage{amsmath}
\usepackage{euscript}
\usepackage{enumerate}
\usepackage{graphics}
\usepackage{color}
\usepackage{pdfsync}
\newtheorem{theorem}{Theorem}[section]
\newtheorem{lemma}[theorem]{Lemma}
\newtheorem{note}[theorem]{Note}
\newtheorem{prop}[theorem]{Proposition}
\newtheorem{cor}[theorem]{Corollary}

\newtheorem*{Theorem1'}{Theorem 1'}

\theoremstyle{definition}

\theoremstyle{remark}

\numberwithin{equation}{section}

\newcommand \N{{\widetilde{N}}}
\newcommand \C{{\mathbb C}}
\renewcommand \O{{\mathcal O}}

\newcommand \ind{{\mathrm {ind}}}

\newcommand \GL{{\mathrm {GL}}}
\newcommand \Sp{{\mathrm {Sp}}}
\newcommand \U{{\mathcal {U}}}

\renewcommand \min{{\mathfrak {n}}}

\newcommand \m{{\mathfrak {m}}}

\renewcommand \i{{\mathfrak {i}}}
\renewcommand \j{{\mathfrak {j}}}
\renewcommand \k{{\mathfrak {k}}}
\renewcommand \r{{\mathfrak {r}}}
\renewcommand \P{{\mathcal P}}

\newcommand \al{{\alpha}}

\begin{document}

\title[Clifford theory of Weil representations]{Clifford theory of Weil\\ representations of unitary groups}

\author{Moumita Shau}
\address{Department of Mathematics and Statistics, University of Regina, Regina, Canada, S4S 0A2}
\email{tuktukishau@gmail.com}

\author{Fernando Szechtman}
\address{Department of Mathematics and Statistics, University of Regina, Regina, Canada, S4S 0A2}
\email{fernando.szechtman@gmail.com}
\thanks{The second author was supported in part by an NSERC discovery grant}



\begin{abstract} Let $\O$ be an involutive discrete valuation ring with residue field of characteristic not 2. Let
$A$ be a quotient of $\O$ by a nonzero power of its maximal ideal and let $*$ be the involution that $A$ inherits from $\O$.
We consider various unitary groups $\U_m(A)$ of rank $m$ over $A$,
depending on the nature of $*$ and the equivalence type of the underlying hermitian or skew hermitian form. Each group $\U_m(A)$ gives rise to a Weil representation.

In this paper we give a Clifford theory description of all irreducible components of the Weil
representation of $\U_m(A)$ with respect to all of its abelian congruence subgroups and a third of its nonabelian congruence subgroups.
\end{abstract}

\maketitle

\section{Introduction}

Since their inception \cite{W}, Weil representations have been studied by several authors. One line of investigation was concerned with symplectic
and unitary groups over finite fields \cite{Ge, Go, H, T, TZ}. The Weil representations of these groups have few irreducible constituents. This feature, however, changes if one replaces fields by rings. The decomposition problem of the Weil representation of symplectic and various unitary groups over suitable finite rings was studied in \cite{CMS1,S, HS, HSS}.

In this paper we give a Clifford theory description of all irreducible components of the Weil
representation of a unitary group $\U_m(A)$ with respect to all of its abelian congruence subgroups and about a third of its nonabelian congruence subgroups.
As indicated below, $\U_m(A)$ is a typical member of five distinct families of groups. Our treatment
is uniform, although each family often requires special attention. We have two general theorems valid in all cases. We require two theorems, depending on the nature of the congruence subgroup: abelian or nonabelian.
In order to state such theorems, additional notation and background will first be established.

Let $\O$ be discrete valuation ring with involution -possibly trivial- having residue field of characteristic not 2 and let
$A$ be a quotient of $\O$ by a nonzero power of its maximal ideal. Then $A$ inherits an involution, say $*$, from $\O$
and we let $R$ stand for the fixed ring of $*$. Three cases arise (see \cite[Proposition 5]{CQS}):

$\bullet$ symplectic: $*$ is trivial, that is, $A=R$.

$\bullet$ unramified: $A=R\oplus \theta R$, where $\theta$ is a unit of $A$ and $\theta^*=-\theta$.

$\bullet$ ramified: $A=R\oplus \pi R$, where $A\pi$ is the maximal ideal of $A$ and $\pi^*=-\pi$.


The ramified case further divides into two cases, odd or even, depending on whether the nilpotency degree of $\pi$ is odd, $2\ell-1$, or even,~$2\ell$.

In all cases, $A$ and $R$ are finite, commutative, principal, local rings of
odd characteristic. Let $\r=A\pi$ and $\m=Rp$ stand for their
maximal ideals, so that $\m=R\cap\r$, and let $F_q=R/\m$ be the
residue field of $R$. Then $A/\r\cong F_q$ in the symplectic and
ramified cases, and $A/\r\cong F_{q^2}$ in the unramified case. We
choose $\pi$ and $p$ so that $\pi=p$ in the symplectic and
unramified cases, and $\pi^2=p$ in the ramified case. The minimal
ideal of $A$ will be denoted by $\min$.

We have $A=R\oplus S$, where $S$ is the additive group of all skew hermitian elements of $A$. In the unramified case, $S=R\theta$
and $\{1,\theta\}$ is an $R$-basis of $A$. In the ramified case, $S=R\pi$, but $\{1,\pi\}$ is an $R$-basis of $A$ in the even case only.
In the ramified odd case the annihilator of $\pi$ in $R$ is $Rp^{\ell-1}$. This is true even in the extreme case when $\ell=1$,
which is the symplectic field case $A=R=F_q$.

Let $V$ be a free $A$-module of finite rank $m\geq 1$. If $*$ is unramified, all nondegenerate skew hermitian (resp. hermitian)
forms defined on $V$ are equivalent \cite[\S 3]{CHQS} and we can pass from one type of form to the other one through multiplication by $\theta$;
both forms yield the same unitary group. If $*$ is ramified (resp. symplectic), all nondegenerate skew hermitian forms defined on $V$ are equivalent and $m=2n$ must be even \cite[\S 2]{CS},
whereas there are two types of nondegenerate hermitian forms defined on $V$ \cite[\S 3]{CHQS}.
If $m\geq 2$ all forms above are isotropic, that is, they admit an isotropic basis vector, except when $m=2$, $*$ is ramified and the form
is hermitian of one of the two types indicated above \cite[\S 3]{CHQS}.

We consider an $R$-linear map $d:A\to R$, namely the projection of $A$ onto $R$ in the symplectic, unramified, and ramified odd cases, and $d(r+s\pi)=s$ in the ramified even case. Let $h:V\times V\to A$ be a nondegenerate hermitian or skew hermitian form. More precisely, we let $h$ be skew hermitian in the symplectic, unramified, and ramified odd cases,
whereas $h$ is hermitian in the ramified even case. We next define an alternating
$R$-bilinear form $f:V\times V\to R$, where $V$ is viewed here as an $R$-module, by $f=d\circ h$. Our definition of $d$ ensures that $f$
is nondegenerate (see Corollary \ref{nondeg}).
Let $\U_m(A)=\U(V)$ stand for the unitary group associated to $h$ and let $\Sp(V)$ be the symplectic group associated to $f$.
By definition, $\U(V)$ is a subgroup of $\Sp(V)$.

It is worth noting that the isomorphism types of the various unitary groups above are quite sensitive to changes of the given parameters. Indeed,
suppose first $m$ is even. If $*$ is symplectic, then $\U(V)=\Sp(V)$ is an extension of the symplectic group $\Sp_m(F_q)$.
If $*$ is unramified then $\U(V)$ is an extension of the unitary group ${\mathrm U}_m(F_{q^2})$. If $*$ is ramified odd, then $\U(V)$
is an extension of the symplectic group $\Sp_m(F_q)$. If $*$ is ramified even, then $\U(V)$ is extension of an orthogonal group ${\mathrm O}_m(F_q)$,
which depends on the type of the underlying hermitian form. Formulas for the orders of all of these groups $\U(V)$ based on the orders of $\Sp_m(F_q)$,
${\mathrm U}_m(F_{q^2})$ and
${\mathrm O}_m(F_q)$ (as given in \cite{G}) can be found in \cite[\S 5]{CS} and \cite[\S 5]{CHQS}, which allows us to see all five unitary groups $\U(V)$ listed above have different orders. Suppose next $m$ is odd. The symplectic and ramified odd cases cease to exist and in the ramified even case the two types of hermitian forms are linked by a unit from $R^\times\setminus R^{\times 2}$, thus yielding the same unitary group. Using \cite[\S 5]{CHQS} and \cite{G} we see that
the remaining two unitary groups $\U(V)$ have different orders.




The Heisenberg group $H(V)$ has underlying set $R\times V$, with multiplication
$$
(r,u)(s,v)=(r+s+f(u,v),u+v).
$$
We identify the center $Z(H(V))=(R,0)$ of $H(V)$ with $R^+$. Note that $\Sp(V)$ acts on $H(V)$ by means of automorphisms via
$$
{}^g (r,u)=(r,gu).
$$

We fix a linear character $\lambda:R^+\to\C^*$ that is primitive, in the sense that its kernel contains no ideals of $R$ but $(0)$.
It gives rise to a linear character $\mu:A^+\to\C^*$, given by
$$
\mu(a)=\lambda(d(a)),\quad a\in A.
$$

Now $H(V)$ has a unique representation $\mathfrak S:H(V)\to\GL(X)$, up to isomorphism, lying over $\lambda$ \cite[\S 2]{CMS1}. It is called Schr${\rm\ddot{o}}$dinger representation of type $\lambda$ and its degree is equal to $\sqrt{|V|}$. It is clear that $\mathfrak S$ is $\Sp(V)$-invariant. In fact \cite[\S 3]{CMS1}, there is a representation $W:\Sp(V)\to\GL(X)$, called Weil representation of type $\lambda$, such that
$$
W(g)\mathfrak S(k)W(g)^{-1}=\mathfrak S({}^g k),\quad g\in\Sp(V), k\in H(V).
$$

We are concerned with $X$ as a $\U(V)$-module. Consider the $\U(V)$-submodule $Bot$ of $X$ formed
by all fixed points of $(0,\min V)$ in $X$, and set
$$
Top=X/Bot.
$$
We know from \cite{CMS1, S, HS, HSS} that $Bot$ is a Weil module for a unitary group $\U_{m}(\overline{A})$, where
$\overline{A}$ is a quotient of~$A$, so we may restrict attention to $Top$.


Given an ideal $\k$ of $A$, the congruence subgroup $\Omega(\k)$ is defined by
$$
\Omega(\k)=\{g\in \U(V)\,|\, gv\equiv v\mod \k V\}.
$$
This is a normal subgroup of $\U(V)$, namely the kernel of the
reduction homomorphism $\U_{m}(A)\to \U_{m}(A/\k)$. That this is
an epimorphism can be found in \cite[Theorem 5.2]{CHQS} in the unramified and ramified even cases, and in \cite[Theorem 4.1]{CS} as well as
in \cite[Theorem 10.1]{CS} in the symplectic and ramified odd cases. It
is easy to see that $\Omega(\k)$ is abelian if and only if
$\k^2=(0)$.

Our main goal is to describe each irreducible component of $Top$, say~$P$, via Clifford theory with respect to all congruence subgroups $\Omega(\k)$ of~$\U(V)$ such that $\k^3=(0)$.
We aim to determine an irreducible $\Omega(\k)$-submodule $M$
of $P$; the stabilizer $G$ of $M$ in $\U(V)$; the $M$-homogeneous component $C$ in $P$, which
is an irreducible $G$-module; the dimension of~$C$; and the actual action of $G$ on $C$. We achieve these goals, thereby giving
a Clifford theory description of the form $P=\ind_{G}^{\U(V)} C$ of each irreducible component $P$ of $Top$ with respect to every congruence
subgroup $\Omega(\k)$ of $\U(V)$ such that $\k^3=(0)$.

A similar result was achieved in \cite{CMS2} for all irreducible components of the Weil module of $\Sp(V)$
with respect to every congruence subgroup. In the ramified case of even type, a Clifford theory for the irreducible constituents of $Top$ with respect to the single congruence
subgroup $\Omega(\r^\ell)$ was developed in \cite{HS}.

When $\k=(0)$, $\k=A$ or when $\U(V)$ is abelian, Clifford theory is trivial. For this reason,
we will assume from now on that $A$ is not a field and $m>1$.

We fix throughout a nonzero ideal $\i$ of $A$ with square $(0)$ and annihilator $\j$. Clearly $\i\subseteq \j$. We also fix
the $A$-submodule $U=\i V$ of $V$. It follows from Lemmas \ref{perp} and \ref{ij} that
$U^\perp=\j V$, with respect to both $h$ and
$f$.  In particular, $U\subseteq U^\perp$, that is, $U$ is totally isotropic. Observe as well that $U$ and $U^\perp$ are $\U(V)$-invariant.

Since $f$ induces a nondegenerate form on $U^\perp/U$, there is a unique -up to isomorphism- irreducible representation $S':H(U^\perp)\to\GL(Z)$
where $(R,U)$ acts via $\lambda$. We refer to $S'$ as a Schr${\rm\ddot{o}}$dinger representation of type $\lambda$. The nondegenerate case again yields a corresponding Weil representation $\Sp(U^\perp)\to\GL(Z)$ of type~$\lambda$ \cite[\S 3]{CMS2}, and we let
$W':\U(U^\perp)\to\GL(Z)$ stand for its restriction to $\U(U^\perp)$. Thus $Z$ is a $H(U^\perp)\rtimes \U(U^\perp)$-module, and therefore a $H(U^\perp)\rtimes \U(V)$-module via the restriction map $\U(V)\to \U(U^\perp)$, given by $g\mapsto g|_{U^\perp}$.

Consider the induced $H(V)\rtimes \U(V)$-module
$$
X=\ind_{H(U^\perp)\rtimes \U(V)}^{H(V)\rtimes \U(V)} Z.
$$
Then the restriction of $X$ to $H(V)$ is isomorphic to $\ind_{H(U^\perp)}^{H(V)} Z$,
which is a Schr${\rm\ddot{o}}$dinger module of $H(V)$ of type $\lambda$ by \cite[Proposition 3.3]{CMS2}. It follows that the restriction of $X$ to $\U(V)$ is a Weil module of type $\lambda$. 

In order to study the Weil module $X$ for $\U(V)$, we use, as in \cite[\S 4]{CMS2}, Mackey subgroup theorem, this time with respect to
the subgroups $H$ and $K$ of $H(V)\rtimes \U(V)$ given by $$K=\U(V), \quad H=H(U^\perp)\rtimes \U(V).$$

Arguing as in \cite[\S 4]{CMS2} we see that if $T$ is a system of representatives for the $\U(V)$ orbits of $V/U^\perp$ and
$$
G_t=\{g\in\U(V)\,|\, gt\equiv t\mod U^\perp\},
$$
then the $\U(V)$-module $X$ admits the following decomposition
$$
X\cong \underset{t\in T}\bigoplus \ind_{G_t}^{\U(V)} Z_t,
$$
where $Z_t=Z$ as vector spaces, and $Z_t$ affords the representation, say $W_t:G_t\to\GL(Z_t)$, given by
\begin{equation}
\label{formu}
W_t(g)=S'(f(gt,t),gt-t)W'(g|_{U^\perp}),\quad g\in G_t.
\end{equation}

Since $\j\i=(0)$, the congruence subgroup $\Omega(\i)$ acts trivially on $\j V$. As $S'(0,u)=1_Z$ for all $u\in U$, (\ref{formu})
implies that given any $t\in T$, the congruence subgroup $\Omega(\i)$ acts on $Z_t$ via $W_t$ through the function $\alpha_t:\Omega(\i)\to\C^*$
defined by
\begin{equation}
\label{formu2}
\alpha_t(g)=\lambda(f(gt,t))=\mu(h(gt,t)),\quad g\in \Omega(\i).
\end{equation}
As $W_t$ is a group homomorphism, so must be $\alpha_t$. In fact, it is easy to verify directly that (\ref{formu2}) defines a linear character of $\Omega(\i)$ for all $t\in V$.

We have a group homomorphism, called the norm map, $A^\times\to R^\times$, given
by $a\mapsto aa^*$. Its kernel will be denoted by $N$. Thus
$$
N=\{a\in A^\times\,|\, aa^*=1\}.
$$
We view $N$ as a central subgroup of $\U(V)$ via $a\mapsto a\cdot
1_V$. It is easy to see that $N$ is in fact the center of $\U(V)$. Let $\widehat{N}$ be the group of all linear characters $N\to\C^*$. Given $\phi\in \widehat{N}$, let $Top(\phi)$ be the corresponding $N$-eigenspace of $Top$. It is a $\U(V)$-submodule of $X$.

A vector $v\in V$ is said to be primitive if $v$ belongs to an $A$-basis of~$V$. The set of all primitive vectors of $V$ will be denoted by $\P$.
Thus, $\P=V\setminus \r V$.

Let $\N$ be the set of all $\phi\in\widehat{N}$ that extend $\alpha_t|_{N\cap (1+\i)}$
for some $t\in T\cap\P$ (this $t$ is necessarily unique by Lemma \ref{imp}).

With this notation, we have the following results.

\begin{theorem}\label{mainabelian} (Abelian Clifford Theory)

(a) We have $\N=\widehat{N}$, except when $m=2$ with $h$ hermitian and non isotropic, in which case
$\N$ is a proper subset of $\widehat{N}$ of cardinality $2(q^\ell-q^{\ell-1})$.

(b) The top layer of the Weil module admits the following multiplicity free decomposition into
irreducible $\U(V)$-modules:
$$
Top\cong \underset{\phi\in\N}\bigoplus Top(\phi).
$$

(c) The top layer of the Weil module admits the following Mackey decomposition into $\U(V)$-submodules:
$$
Top\cong \underset{t\in T\cap\P}\bigoplus \ind_{G_t}^{\U(V)} Z_t,
$$
where $Z_t$ affords the representation (\ref{formu}) and $\dim Z_t=\sqrt{|U^\perp/U|}$.

(d) Let $t\in T\cap\P$. Then the $\U(V)$-module $\ind_{G_t}^{\U(V)} Z_t$ admits the following multiplicity free decomposition into
irreducible components:
$$
\ind_{G_t}^{\U(V)} Z_t\cong \underset{\phi\in E_t}\bigoplus Top(\phi),
$$
where $E_t$ consists of all $|N/N\cap(1+\i)|$ linear characters $\phi:N\to\C^*$ that extend $\alpha_t|_{N\cap (1+\i)}$.

(e) Let $\phi\in \N$. Then there is one and only one $t\in T\cap\P$ such that the linear character
$\alpha_t$ of $\Omega(\i)$ enters $Top(\phi)$. Moreover, the
stabilizer of $\alpha_t$ in $\U(V)$ is $G_t N$, and
$$
Top(\phi)\cong \ind_{G_t N}^{\U(V)} Z_t(\phi),
$$
where $Z_t(\phi)$ is the eigenspace of $Z_t$ for the subgroup $N\cap G_t=N\cap (1+\j)$ of $G_t$ corresponding to the linear character $\phi|_{N\cap (1+\j)}$
(that is, $Z_t(\phi)=e_\phi Z_t$, where $e_\phi$ is the idempotent of $N\cap (1+\j)$ associated to $\phi|_{N\cap (1+\j)}$).
Moreover, $Z_t(\phi)$ is an irreducible $G_t N$-module of dimension
\begin{equation}
\label{dimzeta}
\dim Z_t(\phi)=\frac{\sqrt{|U^\perp/U|}}{|N\cap (1+\j)/N\cap (1+\i)|}.
\end{equation}
\end{theorem}

Note that $Top(\phi)$ may be $(0)$ for some $\phi\in\widehat{N}$ and Theorem \ref{mainabelian} explains precisely when this happens.
Observe also that not all nonzero $Top(\phi)$ have the same dimension. A detailed discussion of this can be found in~\S\ref{prueba3}, based on Theorem \ref{mainabelian}
and the sizes of the $\U(V)$-orbits of~$\P$. Note as well that
while $G_t$-module $Z_t$ is irreducible in both \cite{CMS2} and \cite{HS}, this is no longer true in general, as evidenced by (\ref{dimzeta}).

Our tools to analyse the nonabelian case of Clifford theory are Theorem \ref{mainabelian} and a thorough investigation of the group homomorphism
$\Omega(\j)\to \j V/\i V$, given by $g\mapsto gt-t+\i V$, and found in \S\ref{homomega}. This allows us to transfer the weight of the problem from the unitary group to the Heisenberg group.

\begin{theorem}\label{mainnonabelian} (Nonabelian Clifford Theory)

Suppose $\j^2\subseteq \i$. Given $\phi\in \N$, let $t$ be the unique element of $T\cap\P$ such that the linear character
$\alpha_t$ of $\Omega(\i)$ enters $Top(\phi)$. Then

(a) $Z_t(\phi)$ remains irreducible upon restriction to $\Omega(j)$.

(b) The stabilizer of $Z_t(\phi)$ is $G_t N$ and $Top(\phi)\cong \ind_{G_t N}^{\U(V)} Z_t(\phi)$.

(c) Let $H_t$ be the subgroup of $H(U^\perp)$ generated by $(0,U)$, $(R,0)$ and all $(0,gt-t)$ such that $g\in\Omega(\j)$. Then the restriction
of $S'$ to $H_t$ leaves $Z_t(\phi)$ invariant and $Z\cong\ind_{H_t}^{H(U^\perp)} Z_t(\phi)$. In particular, $Z_t(\phi)$ is an irreducible
$H_t$-module whose stabilizer is $H_t$ itself, and whose index in $H(U^\perp)$ is precisely $|N\cap (1+\j)/N\cap (1+\i)|$.

(d) Let $\Gamma:\Omega(\j)\to \j V/\i V$ be defined by $\Gamma(g)=gt-t+\i V$. Then $\Gamma$ is a group homomorphism and
$[\j V/ \i V:\Gamma(\Omega(\j))]=|N\cap (1+\j)/N\cap (1+\i)|$.
\end{theorem}

\section{Preliminaries}

A vector $v\in V$ is said to be isotropic if $h(v,v)=0$. We say that $h$ is isotropic if $V$ has a primitive isotropic vector.
By a hyperbolic plane we mean a free submodule $E$ of $V$ having rank 2 admitting a basis $\{u,v\}$ formed by isotropic vectors $u,v$ such that $h(u,v)=1$. Such a basis of~$E$ is said to be hyperbolic.

We set $\varepsilon=-1$ if $h$ is skew hermitian and $\varepsilon=1$ if $h$ is hermitian. Given $a\in A$ and a pair of isotropic orthogonal vectors $u,v\in V$ we define the Eichler transformation
$\rho_{a,u,v}$ of $V$ to be the element of $\U(V)$ given by
\begin{equation}
\label{eichler}
\rho_{a,u,v}(x)=x+ah(u,x)v-\varepsilon a^* h(v,x)u,\quad x\in V.
\end{equation}
In particular, if $a\in A$ satisfies $a^*=-\varepsilon a$ and $u\in V$ is isotropic then the unitary transvection $\tau_{a,u}=\rho_{a/2,u,u}$
is the element of $\U(V)$ defined by
$$
\tau_{a,u}(x)=x+ah(u,x)u,\quad x\in V.
$$
Note that if $\k$ is an ideal of $A$ and $a\in\k$ then $\rho_{a,u,v},\tau_{a,u}\in\Omega(\k)$.

\begin{lemma} Every ideal of $A$ is $*$-invariant. Moreover, if $\k$ is an ideal of $A$ then $\k=R\cap\k\oplus S\cap\k$.
\end{lemma}

\begin{proof} It is clear that $\r$ is a $*$-invariant ideal of $A$. Since every ideal of $A$ is a power of $\r$, it follows that every
ideal of $A$ is $*$-invariant.

Let $a\in\k$. Then $a=r+s$, where $r\in R$ and $s\in S$. Given that $r-s=(r+s)^*=a^*\in\k$, it follows that $2r,2s\in\k$, whence $r,s\in\k$.
\end{proof}

\begin{lemma}\label{perp} Given an $A$-submodule $P$ of $V$, set
$$
P^{\perp} = \{v \in V\,|\, f(v,u) = 0\text{ for all }u\in P \},
$$
and
$$
\quad P^{\dagger} = \{v \in V\,|\, h(v,u) = 0\text{ for all }u\in P \}.
$$
Then $P^\perp=P^\dagger$.
\end{lemma}

\begin{proof} Clearly, $P^{\dagger} \subseteq P^{\perp}$. Let $v \in P^{\perp}$, so that
$f(v,u)=0$ for all $u\in P$.

Suppose first $*$ is ramified of odd (resp. even) type. Let $u \in P$. Then $h(v,u)\in S$ (resp. $h(v,u)\in R$)
and therefore $\pi h(v,u)\in R$ (resp. $\pi h(v,u)\in S$). But $\pi h(v,u)=h(v,\pi u)$ and $\pi u\in P$, so
$$
\pi h(v,u)=h(v,\pi u) \in R \cap S = {0}.
$$
We infer that $h(v,u) \in \min \cap S = {0}$ (resp. $h(v,u) \in \min \cap R = {0}$), which implies
$v \in P^{\dagger}$.

Suppose next $*$ is unramified. Let $u \in P$. Then $h(v,u)\in S$ and therefore $\theta h(v,u)\in R$. But $\theta h(v,u)=h(v,\theta u)$ and $\theta u\in P$, so
$$
\theta h(v,u)=h(v,\theta u) \in R \cap S = {0}.
$$
Since $\theta$ is a unit in $A$, we infer $h(v,u)=0$, whence $v \in P^{\dagger}$.

The symplectic case being obvious, the proof is complete.
\end{proof}

\begin{cor}\label{nondeg} The skew symmetric form $f$ is nondegenerate.
\end{cor}

\begin{proof} Apply Lemma \ref{perp} with $P=V$.
\end{proof}

\begin{lemma}\label{ij} We have $(\i V)^\perp=\j V$.
\end{lemma}

\begin{proof} It is clear that $\j V\subseteq (\i V)^\perp$. Moreover, since $|\i||\j|=A$, we have
$$
|\i V|\times |\j V|=|V|.
$$
On the other hand, \cite[Lemma 2.1]{CMS2} shows that
$$
|\i V|\times |(\i V)^\perp|=|V|,
$$
so $|(\i V)^\perp|=|\j V|$ and therefore $(\i V)^\perp=\j V$.
\end{proof}

\begin{lemma}\label{nis} We have $N\cap (1+\i)=1+(S\cap\i)$.
\end{lemma}

\begin{proof} Since $\i$ is a $*$-invariant ideal of $A$ of square $(0)$, for $z\in\i$ we have $1+z\in N$
if and only if $z\in S$.
\end{proof}

\section{Stabilizer of $\alpha_t$}\label{lem}

\begin{prop}
	\label{alv=alw}
	Suppose $v$, $w \in V $ are primitive and satisfy $\alpha _v = \alpha _w$. Then there exist $z \in N$ and a vector $u_0 \in U^\perp$ such that $w= zv + u_0$
\end{prop}

\begin{proof} The proof given in \cite[Proposition 3.1]{HS} remains valid, mutatis mutandis. Thus, we restrict to indicate the main steps
and refer the reader to \cite[Proposition 3.1]{HS} for details.

Let $a\in \i $ and $z_1$,$z_2 \in V$. Since $\i^2=(0)$, we readily see that $\rho _{a,z_1,z_2}$, as defined in (\ref{eichler}), belongs
to $\U(V)$ and hence to $\Omega (\i)$. From $\alpha _v(\rho_{a,z_1,z_2}) = \alpha _w(\rho_{a,z_1,z_2})$, we obtain
$$
h(v,z_1)h(z_2,v)\equiv h(w,z_1)h(z_2,w) \mod \j,\quad z_1,z_2\in V.
$$
From this we deduce that $v,w$ are linearly dependent and, in fact, that
$$
w= tv + \pi ^s u
$$
for some $t \in A^{\times}$, $u\in V$ and $0 <s \leq e$, where $e$ is the nilpotency degree of $\pi$ and $s$ is chosen as large as possible.
The choice of $s$ ensures that $u_1=\pi^s u\in\j V$ and $tt^*\equiv 1\mod \j$. This readily implies $w= zv + u_0$, for some $z\in N$ and
$u_0 \in \j V$.
\end{proof}

\begin{prop}\label{inertia} Let $v\in\P$. Then the stabilizer $\alpha_v$ in $\U(V)$ is $NG_v$.
\end{prop}

\begin{proof} Since $N$ is a central subgroup of $\U(V)$, it is clear that $N$ stabilizes $\alpha_v$. We claim that
$G_v$ stabilizes $\alpha_v$. Indeed, note first of all that for $x \in \U(V)$ we have ${}^x \alpha_v=\al_{xv}$.
Let $x\in G_v$ and $g\in\Omega(\i)$. Then $xv=v+u$ for some $u\in U^\perp$, so
$$ ^{x}\alpha_v(g)= \alpha_{xv}(g)=\alpha_{v+u}(g) = \mu(h(g(v+u),(v+u)).$$
Since $\Omega(\i)$ acts trivially on $U^\perp$, we have $gu=u$. Moreover, since $h(U,U^\perp)=0$, it follows that $h(gv,u)=h(v,u)$.
Hence
$$^{x}\alpha_v(g)=\mu(h(gv,v))\mu(h(v,u)+h(u,v))\mu(h(u,u)).$$
Our definition of $\mu$ readily implies that $\mu(h(v,u)+h(u,v))=1$ and $\mu(h(u,u))=1$, so $^{x}\alpha_v(g)=\alpha_v(g)$,
as claimed.

Suppose, conversely, that $x\in\U(V)$ stabilizes $\alpha_v$. Then $\al _{xv} = \al _v $. It follows from Proposition \ref{alv=alw} that $xv=zv + u_0$ for some $z \in N$ and $u_0 \in U^\perp$. Therefore $z^{-1} x \in G_v$, so $x \in NG_v$.
\end{proof}


\begin{lemma}\label{imp} Let $v,w\in\P$. Suppose that $\alpha_v|_{N\cap (1+i)}=\alpha_w|_{N\cap (1+i)}$.
Then $v,w$ are in the same $\U(V)$-orbit modulo $U^\perp$.
\end{lemma}

\begin{proof} In the symplectic case any two primitive vectors of $V$ are in the same $\Sp(V)$-orbit,
so the result is obvious in this case. We next deal with the other three cases.

By \cite[Theorem 3.1]{CS} and \cite[Theorem 4.1]{CHQS}, $v,w$ are in the same $\U(V)$-orbit modulo $U^\perp$ if and only if
$h(v,v)\cong h(w,w)\mod \j$.

Now, the very definition of $\mu$ yields
\begin{equation}
\label{df}
\mu(h(v,v))=\mu(h(w,w)).
\end{equation}
Let $z\in S\cap\i$. Then $1+z\in N\cap(1+\i)$ by Lemma \ref{nis}. Our hypothesis $\alpha_v|_{N\cap (1+i)}=\alpha_w|_{N\cap (1+i)}$ together
with (\ref{df}) yield
$$
\mu(h(zv,v))=\mu(h(zw,w)).
$$
Since $z^*=-z$, this is equivalent to
$$
\mu(z(h(v,v)-h(w,w)))=1.
$$

In the unramified or ramified odd cases, $h(v,v) - h(w,w) \in S$, so $z(h(v,v) - h(w,w)) \in R$. Thus $Rz(h(v,v) - h(w,w))$ is an ideal
of $R$ contained in the kernel of $\lambda$. Since $\lambda$ is primitive, we infer
\begin{equation}
\label{df2}
z(h(v,v) - h(w,w))=0.
\end{equation}

In the ramified even case, we have $h(v,v) - h(w,w) \in R$, and therefore $z(h(v,v) - h(w,w))\in S$. Thus $d(Rz(h(v,v) - h(w,w))) $ is an ideal of~$R$
contained in kernel of $\lambda$, so $d(z(h(v,v) - h(w,w))) =0$, whence
\begin{equation}
\label{df3}
z(h(v,v) - h(w,w)) \in R \cap S =(0).
\end{equation}

We have $\i = A\pi^i$ and $\j = A\pi^j$. In the unramified case, we take $z=\pi^i\theta\in S\cap\i$ in (\ref{df2}) and deduce
$h(v,v) - h(w,w)\in \mathrm{Ann}_A \pi^i=\j$. Suppose henceforth $*$ is ramified.

In the ramified odd (resp. even) case, assume first $i$ is odd. We take $z=\pi^i\in S\cap\i$ in (\ref{df2}) (resp. (\ref{df3}))
and deduce, as before, that $h(v,v) - h(w,w)\in \mathrm{Ann}_A \pi^i=\j$. Assume next $i$ is even, in which case $j$ is odd (resp. even).
Taking $z=\pi^{i+1}\in S\cap\i$ in (\ref{df2}) (resp. (\ref{df3})), we infer
$h(v,v) - h(w,w)\in \mathrm{Ann}_A \pi^{i+1}=A\pi^{j-1}$. But $h(v,v) - h(w,w)$ is in $S$ (resp. in $R$) and $j-1$ is even (resp. odd), so $h(v,v) - h(w,w)$ is in $A\pi^{j-1}\cap S=A\pi^{j}\cap S$ (resp. in $A\pi^{j-1}\cap R=A\pi^{j}\cap R$).
\end{proof}

\begin{lemma}\label{bot} $Bot$ is the sum of all $\ind_{G_t}^{\U(V)} Z_t$ such that $t\in T$ but $t\notin \P$.
\end{lemma}

\begin{proof} Recall from the Introduction the induced $H(V)\rtimes \U(V)$-module
$$
X=\ind_{H(U^\perp)\rtimes \U(V)}^{H(V)\rtimes \U(V)} Z.
$$
Let $M$ a transversal for $U^\perp$ in $V$ containing $T$. Then the set of all $(0,v)$, with $v\in V$,
is a transversal for $H(U^\perp)\rtimes \U(V)$ in $H(V)\rtimes \U(V)$, so we have
the following vector space decomposition,
$$
X=\underset{v\in M}\bigoplus (0,v) Z,
$$
where each summand is an $H(V)$-submodule of $X$. Given $t\in T$, the sum of all $(0,s) Z$, where $s$ runs through
the $\U(V)$-orbit of $t$, is a copy $\ind_{G_t}^{\U(V)}Z_t$ inside of $X$.

We first show that $(0,\min V)$ fixes every point of the sum of all $\ind_{G_t}^{\U(V)} Z_t$, $t\in T\cap\r V$, which means that
$(0,\min V)$ fixes every element of $(0,v) Z$ with $v\in M\cap\r V$. Now, if $u\in \min V$ and $v\in M\cap\r V$ then $(0,u)(0,v)=(0,v)(0,u)$ since $\min\r=0$ and $(0,u) z=z$ for every $z\in Z$, since the
subgroup $(0,\min)\subseteq (0,U)$ of $H(U^\perp)$ acts trivially on the Schr${\rm\ddot{o}}$dinger module of $H(U^\perp)$ of type $\lambda$.
This shows that
$$
\underset{t\in T\cap \r V}\bigoplus \ind_{G_t}^{\U(V)} Z_t=\underset{v\in M\cap \r V}\bigoplus (0,v) Z\subseteq Bot.
$$
The left hand side has dimension
$$
|\r V|/|\j V|\times \sqrt{|\j V|/|\i V|}=|\r V|/\sqrt{|\j V|\times |\i V|}=|\r V|/\sqrt{|V|},
$$
and the right hand side has dimension
$$
\sqrt{|\r V|/|\min V|}.
$$
These are the same, since $|\r V||\min V|=|V|$.
\end{proof}

\section{Counting orbits, linear characters and dimensions}\label{counting}

\begin{lemma}\label{lincar} We have $|T\cap\P|=|N\cap (1+\i)|$, except when $h$ is hermitian and non isotropic, when $|T\cap\P|=(1-1/q)|N\cap (1+\i)|$.
\end{lemma}

\begin{proof}  Let $\overline{A}=A/\j$. Then $\overline{V}=V/\j V$ is a free $\overline{A}$-module of rank $m$, endowed with a non degenerate form $\overline{h}$ inherited from $h$.
As indicated in the Introduction, the canonical homomorphism $\U(V)\to\U(\overline{V})$ is surjective. Therefore, $|T\cap\P|$ is equal to the number of $\U(\overline{V})$-orbits
of basis vectors from $\overline{V}$. By \cite[Theorem 4.1]{CHQS} and \cite[Theorem 3.1]{CS},
two basis vectors $\overline{u},\overline{v}\in \overline{V}$ are in the same $\U(\overline{V})$-orbit if and only if $\overline{h}(\overline{u},\overline{u})=\overline{h}(\overline{v},\overline{v})$.

If $*$ is trivial, the only value $\overline{h}(\overline{u},\overline{u})$, with $\overline{u}$ a basis vector, is $\overline{0}$ and
$N\cap (1+\i)=\{\pm 1\}\cap (1+\i)=\{1\}$, and the result follows in this case.

If $*$ is nontrivial, the values $\overline{h}(\overline{u},\overline{u})$, with $\overline{u}$ a basis vector,
are the skew hermitian elements of $\overline{A}$ if $h$ is skew hermitian (this easily follows from \cite[Proposition 2.12]{CS} in
the ramified case and \cite[Theorem 3.5]{CHQS} in the unramified case), the hermitian elements of $\overline{A}$ if $h$ is hermitian
and isotropic (see \cite[Lemma 3.7]{CHQS}), and the hermitian units of $\overline{A}$ if $h$ is hermitian
and non isotropic, in which case $m=2$ (see \cite[Lemma 3.7]{CHQS}).

Thus, we are reduced to showing that the number of skew hermitian elements of $\overline{A}$ is equal to $|N\cap (1+\i)|$ in the unramified
and ramified odd cases, and that the number of hermitian elements (resp. hermitian units) of $\overline{A}$ is equal to $|N\cap (1+\i)|$ (resp. $(1-1/q)|N\cap (1+\i)|$) in the ramified even case.

Now by Lemma \ref{nis}, we have
$$
|N\cap (1+\i)|=|S\cap\i|.
$$
On the other hand, the additive group of all skew hermitian elements of $A/\j$ is $(S+\j)/\j\cong S/S\cap \j$ and the ring of hermitian elements
of $A/\j$ is $(R+\j)/\j\cong R/R\cap \j$. Thus, we are reduced to showing that
$$
|S|=|S\cap\j|\times|S\cap\i|
$$
in the unramified and ramified odd case, and
$$
|R|=|R\cap\j|\times|S\cap\i|,
$$
$$
|(R/R\cap \j)^\times|=(1-1/q)|S\cap\i|
$$
in the ramified even case.

Suppose first $*$ is unramified. Then $A=R\oplus R\theta$, where $\theta^*=-\theta$ is a unit. For every ideal $\k$ of $A$ we have a bijection $R\cap\k\to S\cap\k$, given
by $r\mapsto r\theta$, so $|S\cap \k|=|R\cap\k|$. Moreover, since every ideal of $A$ is of the form $p^i A$, where $p$ is a generator of
the maximal ideal $\m$ of $R$, we see that the annihilator of $R\cap\i$ in $R$ is precisely $R\cap\j$. Therefore,
$$
|S|=|R|=|R\cap\j|\times|R\cap\i|=|S\cap\j|\times|S\cap\i|.
$$

Suppose next that $*$ is ramified of odd type. It is easy to see that, for $0\leq k\leq 2(\ell-1)$, we have
$$
|S\cap A\pi^k|=\begin{cases}
q^{\ell-1-k/2}\text{ if }k\text{ is even},\\
q^{\ell-1-(k-1)/2}\text{ if }k\text{ is odd}.
\end{cases}
$$
Now $\i=A\pi^i$, $\j=A\pi^j$, where $i+j=2\ell-1$. In particular, exactly one element of $\{i,j\}$ is even. Using the above formula for $S\cap A\pi^k$, we readily find that
$$
|S\cap\j|\times|S\cap\i|=q^{\ell-1}=|S|.
$$

Suppose finally that $*$ is ramified of even type. We have $\i=A\pi^i$, $\j=A\pi^j$, where $i+j=2\ell$. Assume first $i$ is even. Then $j$ is also even and
$$
\i=Ap^{i/2}=R p^{i/2}\oplus R p^{i/2}\pi,\quad \j=Ap^{j/2}=R p^{j/2}\oplus R p^{j/2}\pi.
$$
Since $\{1,\pi\}$ is an $R$-basis of $A$, we infer
$$
|S\cap\i|=|R\cap\i|.
$$
Moreover, $R\cap\j=R p^{j/2}$ is the annihilator of $R\cap\i=R p^{i/2}$ in $R$. Therefore,
$$
|R|=|R\cap\j|\times|R\cap\i|=|R\cap\j|\times|S\cap\i|.
$$
Assume finally that $i=2s+1$ is odd. Then $R\cap\j=Rp^{\ell-s}$ and, since $\{1,\pi\}$ is an $R$-basis of $A$, we find that $|S\cap\i|=|Rp^s|$.
Hence
$$
|R|=|Rp^{\ell-s}|\times|Rp^s|=|R\cap\j|\times|S\cap\i|.
$$
This proves $|R|=|R\cap\j|\times|S\cap\i|$, regardless of the parity of $i$. Moreover, we have a canonical group epimorphism $R^\times\to (R/R\cap\j)^\times$ with kernel $1+R\cap\j$, whence
$$
|(R/R\cap\j)^\times|=|R^\times/1+R\cap\j|=(|R|-|\m|)/|R\cap\j|=(1-1/q)\times |R/R\cap\j|.
$$
This proves that
$$
|(R/R\cap\j)^\times|=(1-1/q)\times |S\cap\i|.
$$
\end{proof}

\begin{cor}\label{co} We have $\N=\widehat{N}$, except when $h$ is hermitian and non isotropic, when $\N$ is a proper subset of $\widehat{N}$ of cardinality $(1-1/q)|N|$.
\end{cor}

\begin{proof} By Lemma \ref{imp}, we have
$$
|\N|=|T\cap\P|\times |N/N\cap (1+\i)|.
$$
The result now follows from Lemma \ref{lincar}.
\end{proof}






\begin{lemma}\label{dim} Suppose $h$ is isotropic. Let $\phi\in\N$ and let $t\in\P\cap T$ be the only element such that $\phi$ extends $\alpha_t|_{N\cap (1+\i)}$.
Then
$$
\dim Z_t(\phi)=\frac{\sqrt{|U^\perp/U|}}{|N\cap (1+\j)/N\cap (1+\i)|}.
$$
\end{lemma}

\begin{proof} In the symplectic case there is nothing to do, so we assume henceforth that $*$ is nontrivial.

We claim that the action of $N\cap (1+\j)$ on $Z_t$ is monomial with stabilizers all equal to $N\cap (1+\i)$.
To see this, we find an $R$-submodule $Z$ of $\j V$ satisfying the following conditions:

(C1) $Z$ is $\r$-invariant, and therefore $N\cap (1+\j)$-invariant.

(C2) $((Z+\i V)/\i V)^\perp=(Z+\i V)/\i V$, with respect to the nondegenerate skew symmetric form
that $f$ induces on $\j V/\i V$.

(C3) Given any $w\in \j V$, we have $\j(t+w)\cap (Z+U)=\i(t+w)$.

Since $h$ is isotropic, $t$ belongs to a hyperbolic plane $E$ with hyperbolic basis $\{u_1,v_1\}$ such that $t=cu_1+v_1$ for some $c\in A$.

In the ramified odd case, $E^\perp$ is the direct sum of $n-1$ hyperbolic planes $E_i$ with hyperbolic bases $\{u_i,v_i\}$, and we take $Z$ to be the $\j$-span of $\{u_1,\dots,u_n\}$.

In the unramified and ramified even cases, $E^\perp$ has an orthogonal basis $w_3,\dots,w_{m}$.
Let $e$ be nilpotency degree of $\r$. If $e$ is even, we take $Z=\j u_1\oplus \r^{e/2} E^\perp$. If $e$ is odd, then $*$ is unramified and all $h(w_i,w_i)$ are units in $R\theta$, and we take
$$Z=\j u_1\oplus R p^{(e+1)/2}w_3\oplus\cdots R p^{(e+1)/2}w_m\oplus R p^{(e-1)/2}\theta w_3\oplus\cdots R p^{(e-1)/2}\theta w_m.$$

This produces the desired $Z$. We next set $Q=Z+U$ and let $Y=\C y$ be a one dimensional module for the subgroup
$(R,Q)$ of $H(U^\perp)$, acting on $Y$ via $\lambda$. By (C2), $B=\ind_{(R,Q)}^{H(U^\perp)} Y$ affords a Schr$\mathrm{\ddot{o}}$dinger representation $S':H(U^\perp)\to\GL(B)$ of type $\lambda$. Let $C$ be a transversal for $Q$ in $U^\perp$ and for $v\in C$
let $e_v=(0,v)y\in B$. Then $(e_v)_{v\in C}$ is a basis of $B$ and $H(U^\perp)$ permutes the one dimensional subspaces $\C e_v$, $v\in C$,
as follows:
\begin{equation}
\label{tu}
(0,w)\C e_v=\C e_{u},\quad w\in U^\perp, v\in C,
\end{equation}
where $u$ be the only element of $C$ such that $w+v\equiv u\mod Q$. Moreover, it follows from (C1) that $Q$ is $N\cap (1+\j)$-invariant,
so the argument given in \cite[\S 3]{CMS1} shows that $N\cap (1+\j)$ also permutes the one dimensional subspaces $\C e_v$, $v\in C$, via $W'$.
In fact,
\begin{equation}
\label{tu2}
g\C e_v=\C e_{z},\quad g\in N\cap (1+\j), v\in C,
\end{equation}
where $z$ is the only element of $C$ such that $gv\equiv z\mod Q$. Using (\ref{formu}), (\ref{tu}) and (\ref{tu2}) we see that
the stabilizer in $N\cap (1+\j)$ of a given $\C e_v$ is formed by all $1+a\in N\cap (1+\j)$ such that
$$
(1+a)v+at\equiv v\mod Q.
$$
By (C3), this implies $a\in\i$, so the stabilizer of $\C e_v$ in $N\cap (1+\j)$ is indeed $N\cap (1+\i)$. This proves the claim.

Since $N\cap (1+\i)$ acts on $Z_t$ via $\alpha_t|_{N\cap (1+\i)}$, the above implies that the $N\cap (1+\j)$-module $Z_t$ is the direct sum of
$$
\frac{\dim(Z)}{|N\cap (1+\j)/N\cap (1+\i)|}
$$
copies of  $\ind_{N\cap (1+\i)}^{N\cap (1+\j)} P_t$, where $P_t$ is a one dimensional $N\cap (1+\i)$-module via $\alpha_t|_{N\cap (1+\i)}$.
On the other hand, $\alpha_t|_{N\cap (1+\i)}$ extends to $N\cap (1+\j)$, because $N\cap (1+\i)$ is abelian, so $\ind_{N\cap (1+\i)}^{N\cap (1+\j)} P_t$
is the direct sum of one dimensional $N\cap (1+\j)$-modules, one for each extension of $\alpha_t|_{N\cap (1+\i)}$ to  $N\cap (1+\j)$ (use Gallagher theorem).
The result thus follows.
\end{proof}

\section{Abelian Clifford theory}\label{abel}

\noindent{\it Proof of Theorem \ref{mainabelian}.} (a) This is proven in Corollary \ref{co}.

(b) This is proven \cite[Theorem 5.4]{CMS1}, \cite[Theorem 5.1]{S}, \cite[\S 6 and \S 7]{HS} and \cite[\S 5]{HSS}.

(c) As indicated in the Introduction, we have
$$
Top\cong \underset{t\in T}\bigoplus \ind_{G_t}^{\U(V)} Z_t.
$$
On the other hand, by Lemma \ref{bot}, $Bot$ is $\U(V)$-isomorphic to the direct sum of all $\ind_{G_t}^{\U(V)} Z_t$ such that $t\in T$ but $t\notin \P$. By definition, $Top=X/Bot$, so
$$
Top\cong \underset{t\in T\cap\P}\bigoplus \ind_{G_t}^{\U(V)} Z_t.
$$

(d) and (e) As indicated in the Introduction,
the congruence subgroup $\Omega(\i)$ acts on any given $Z_t$ via $\alpha_t$.
Since $N$ is central in $\U(V)$, it follows that $N\cap (1+\i)=N\cap \Omega(\i)$ acts on $\ind_{G_t}^{\U(V)} Z_t$ via $\alpha_t|_{N\cap (1+\i)}$.

Fix $t\in T\cap\P$. Since $Top$ is multiplicity free, $\ind_{G_t}^{\U(V)} Z_t$  is isomorphic to the sum of some $Top(\phi)$.
All of these summands must satisfy $\phi|_{N\cap (1+\i)}=\alpha_t|_{N\cap (1+\i)}$, so $\phi$ must be an extension of
$\alpha_t|_{N\cap (1+\i)}$ to $N$.

Now $\alpha_t|_{N\cap (1+\i)}$ extends to $|N/N\cap (1+\i)|$ distinct linear characters of $N$.
By Lemma \ref{imp}, if $u,v$ are distinct elements of $T\cap\P$, then $\alpha_u|_{N\cap (1+i)}\neq \alpha_v|_{N\cap (1+i)}$.
Thus for {\em every} one of the $|N/N\cap (1+\i)|$ extensions $\phi$ of $\alpha_t|_{N\cap (1+\i)}$ to $N$, there is a copy
of $Top(\phi)$ inside $\ind_{G_t}^{\U(V)} Z_t$. This proves that $\ind_{G_t}^{\U(V)} Z_t$ is isomorphic to the direct sum of exactly $|N/N\cap (1+\i)|$
irreducible modules, namely all $Top(\phi)$ such that $\phi$ extends $\alpha_t|_{N\cap (1+\i)}$.

Let $\phi_1,\dots,\phi_k$ be the list of distinct linear characters of the group $N\cap (1+\j)=N\cap G_t$
that extend  the linear character $\alpha_t|_{N\cap (1+\i)}$ of $N\cap (1+\i)$. It is easy to see that $k=|N\cap (1+\j)/N\cap (1+\i)|$.
As $Z_t$ is a $G_t$-module, we may consider its eigenspaces $Z_t(\phi_1),\dots,Z_t(\phi_k)$. These are $G_t$-submodules of $Z_t$.
Since $N\cap (1+\i)$ acts on $Z_t$ via  $\alpha_t|_{N\cap (1+\i)}$, it follows that
$$
Z_t=\underset{1\leq s\leq k}\bigoplus Z_t(\phi_s).
$$
Note that $Z_t$ is the direct sum of \emph{all} of its $N\cap G_t$-eigenspaces, but those $\phi$ not extending $\alpha_t|_{N\cap (1+\i)}$ produce
$Z_t(\phi)=(0)$ because $N\cap (1+\i)$ acts on $Z_t$ via $\alpha_t|_{N\cap (1+\i)}$. Now each $Z_t(\phi)=e_\phi Z_t$, where $e_\phi$
is the idempotent of $N\cap G_t$ associated to $\phi$. Since $N$ is central in $\U(V)$, we see that
$$
g e_\phi z=e_\phi g z,\quad g\in G_t, z\in Z_t,
$$
where the action of $g$ and $e_\phi$ on $Z_t$ is in (\ref{formu}). Thus, we know how $G_t$ acts on each $Z_t(\phi_s)$.

Now each $\phi_s$ extends to a linear character of $N$ in $r=|N/N\cap (1+\j)|$ ways. Let
$\psi_{1,s},\dots,\psi_{r,s}$ be these characters. For each $1\leq l\leq r$, we can then consider the $G_t N$-module $Z_t(\psi_{l,s})$,
which is the $G_t$-module $Z_t(\phi_s)$ with the action of $N$ extended from $\phi_s$ to $\psi_{l,s}$. This is well-defined.

Since the action of $G_t$ on $Z_t(\phi_s)$ extends to $G_t N$, we have
$$
\ind_{G_t}^{G_t N} Z_t(\phi_s)\cong \ind_{G_t}^{G_t N} (Z_t(\phi_s)\otimes 1_{G_t})\cong Z_t(\phi_s)\otimes  \ind_{G_t}^{G_t N} 1_{G_t}.
$$
Since $\ind_{G_t}^{G_t N} 1_{G_t}$ is the regular character of the abelian group $G_t N/G_t\cong N/N\cap (1+\j)$ inflated to $G_t N$, it follows that
$$
\ind_{G_t}^{G_t N} Z_t(\phi_s)\cong\underset{1\leq l\leq r}\bigoplus Z_t(\psi_{l,s}).
$$
Now
$$
\begin{aligned}
\ind_{G_t}^{\U(V)} Z_t &\cong\ind_{G_t N}^{\U(V)} \ind_{G_t}^{G_t N} Z_t\\
&\cong\ind_{G_t N}^{\U(V)} \ind_{G_t}^{G_t N} \underset{1\leq s\leq k}\bigoplus Z_t(\phi_s)\\
&\cong\ind_{G_t N}^{\U(V)} \underset{1\leq s\leq k}\bigoplus \ind_{G_t}^{G_t N} Z_t(\phi_s)\\
&\cong\ind_{G_t N}^{\U(V)} \underset{s,l}\bigoplus  Z_t(\psi_{l,s})\\
&\cong\underset{s,l}\bigoplus  \ind_{G_t N}^{\U(V)}  Z_t(\psi_{l,s}).
\end{aligned}
$$
It is clear that $\ind_{G_t N}^{\U(V)}  Z_t(\psi_{l,s})$ is contained in $Top(\psi_{l,s})$. Therefore, either $\ind_{G_t N}^{\U(V)}  Z_t(\psi_{l,s})=(0)$
or $\ind_{G_t N}^{\U(V)}  Z_t(\psi_{l,s})=Top(\psi_{l,s})$. But  $\ind_{G_t}^{\U(V)} Z_t$ is the
direct sum of $|N/N\cap (1+\i)|$ irreducible modules, so equality prevails in all cases.

Thus $\ind_{G_t N}^{\U(V)}  Z_t(\psi_{l,s})=Top(\psi_{l,s})$ is
irreducible. In particular, $Z_t(\psi_{l,s})$ is an irreducible $G_t N$-module, clearly lying
over $\alpha_t$. Since $N$ is central in $\U(V)$, it follows that $Z_t(\phi_s)$ is indeed an irreducible $G_t$-module.
By Proposition \ref{inertia}, the stabilizer of $\alpha_t$ in $\U(V)$ is $G_t N$.

When $h$ is isotropic, the dimension $Z_t(\phi)$ is given in Lemma \ref{dim}. Suppose next $h$ is non isotropic.
Then $m=2$, $h$ is hermitian and $*$ is ramified even. Reduction modulo $\j$ yields a canonical projection $V\mapsto \overline{V}$,
which gives rise to the group epimorphism $\U(V)\to\U(\overline{V})$. The subgroup $G_t$ maps onto the point stabilizer $S_{\overline{t}}$ of
$\overline{t}\in\overline{V}$
in $\U(\overline{V})$. By \cite[Proposition 6.2]{CHQS}, the index $[\U(\overline{V}):S_{\overline{t}}]$ is independent of $\overline{t}$. Now
$$
[\U(V):G_t]=[\U(V)/\Omega(\j):G_t/\Omega(\j)]=[\U(\overline{V}):S_{\overline{t}}],
$$
so $[\U(V):G_t]$ is independent of $t$ and hence of $\phi$. Since
$$
[\U(V):G_t N]=[\U(V):G_t]/[G_t N:G_t]=[\U(V):G_t]/[N:N\cap (1+\j)],
$$
we infer that $[\U(V):G_t N]$ is independent of $t$ and hence of $\phi$. Suppose first $\i=\r^\ell$. Then $Z_t$ is one dimensional, whence
the dimension of $Top(\phi)\cong\ind_{G_t N}^{\U(V)}  Z_t$ is independent of $\phi$. For arbitrary $\i$, since $\dim(Top(\phi))$ and
$[\U(V):G_t N]$ are independent of $\phi$, it follows from $Top(\phi)\cong\ind_{G_t N}^{\U(V)}  Z_t(\phi)$ that $\dim Z_t(\phi)$ is independent of $\phi$,
whence $\dim Z_t(\phi)=\dim(Z)/{|N\cap (1+\j)/N\cap (1+\i)|}$.\qed

\section{A homomorphism $\Omega(\j)\to \j V/\i V$}\label{homomega}

We fix throughout a primitive vector $t\in V$ and assume that $\j^2\subseteq\i$. The condition $\j^2\subseteq\i$
readily yields that the map $\Gamma:\Omega(\j)\to U^\perp/U$ given by $\Gamma(g)=gt-t+U$ is a group homomorphism. We set
$$I=\Gamma(\Omega(\j)),\quad I(m)=[(U^\perp/U):I].$$

We want to prove that $I(m)=|N\cap (1+\j)/N\cap (1+\i)|$. The goal of this section is to prove that
$I(m)\leq |N\cap (1+\j)/N\cap (1+\i)|$. This will yield the desired equality by means of representation theory.

\begin{lemma}\label{m3} Suppose $m>2$. Then
\begin{equation}
\label{idos}
I(m)\leq I(2).
\end{equation}
\end{lemma}

\begin{proof} Since $m>2$, $V$ contains a primitive isotropic vector and a fortiori a hyperbolic plane. This implies that $V$ has a basis
formed by isotropic vectors. Since $t$ is primitive, we deduce the existence of an isotropic vector, say $u$ such that $h(u,t)=1$.

Set $E=\mathrm{span}\{u,t\}$ and $F=E^\perp$. Then $E$ and $F$ are free submodules of $V$ of ranks 2 and $m-2$, respectively, and $V=E\oplus F$.

Suppose first that $F$ has a primitive isotropic vector. Then, as seen above, $F$ is generated as abelian group by isotropic vectors. Thus $\j F$
is generated as abelian group by vectors of the form $aw$, where $a\in\j$ and $w$ is isotropic. Given such a vector $aw$, consider
the Eichler transformation $\rho_{a,u,w}\in\Omega(\j)$. Then
$$
\rho_{a,u,w}(t)-t=aw.
$$
This proves that $(\j F+U)/U$ is included in $I$, which implies (\ref{idos}).

Suppose next that $F$ has no primitive isotropic vector. By \cite[Proposition 2.12]{CS}, $*$ is unramified or ramified of even type. In either case, \cite[Theorem 3.5]{CHQS} implies that $F$ has an orthogonal basis, say $\{w_1,\dots,w_{m-2}\}$. We claim that given any $a\in\j$ and any $1\leq i\leq m-2$, we have
\begin{equation}
\label{itres}
aw_i+U\equiv z+U\mod I
\end{equation}
for some $z\in\j E$, which implies (\ref{idos}).

Let $\{u,v\}$ be a hyperbolic basis of $E$. Since $h(u,t)=1$, we have $t=eu+v$ for some $e\in A$. Let $g$ be the $A$-linear
transformation of $V$ which acts like the identity on the orthogonal complement to $\mathrm{span}\{u,v,w_i\}$ and whose restriction
to $\mathrm{span}\{u,v,w_i\}$ is represented by the matrix $G$ relative to the basis $\{u,v,w_i\}$, where
$$
G=\left(
  \begin{array}{ccc}
    1 & b & c \\
    0 & 1 & 0  \\
    0 & a & 1\\
  \end{array}
\right),
$$
$$
b=-\varepsilon aa^*d/2,\; c=-\varepsilon a^*d,\; d=h(w_i,w_i).
$$
Then $g\in\Omega(\j)$ and $g(t)-t=bu+aw_i$, which gives (\ref{itres}).
\end{proof}

\begin{lemma}\label{sn} The elements of $N\cap (1+\j)$ are exactly those of the form $1+z$, where $z=s^2/2+s$
and $s\in S\cap\j$.
\end{lemma}

\begin{proof} Suppose first $1+z\in N\cap (1+\j)$. Then $z=r+s$, where $r\in R\cap\j$ and $s\in S\cap\j$. Thus
$$
1=(1+r+s)(1+r-s)=1+(r+s)(r-s)+2r=1+r^2-s^2+2r,
$$
so
$$
r^2-s^2=-2r.
$$
But $r^2-s^2=(r+s)(r-s)\in\j^2\subseteq\i$, so $r\in\i$, whence $r^2=0$ and fortiori $r=s^2/2$ and
$$
z=s^2/2+s.
$$
Conversely, for any such $z$,
$$
(1+z)(1+z)^*=1+s^4/4-s^2+2(s^2/2)=1,
$$
because $s^4\in \j^2\j^2\subseteq \i\i=(0)$.
\end{proof}

By Lemma \ref{sn}, given any $s\in S\cap\j$ the element $1+(s^2/2+s)$ belongs to $N\cap (1+\j)=N\cap\Omega(\j)$ and satisfies
\begin{equation}\label{p1}
(1+(s^2/2+s))t-t+U=st+U.
\end{equation}

\begin{lemma}\label{m2} Suppose $m=2$ and $h$ is skew hermitian (resp.  hermitian and isotropic).  Then
$$
I(2)\leq |N\cap (1+\j)/N\cap (1+\i)|.
$$
\end{lemma}

\begin{proof} By \cite[Proposition 2.12]{CS}, $V$ has a hyperbolic basis $\{u,v\}$ (if $h$ is hermitian and isotropic,
the existence of such a basis is readily verified). Given that $t$ is primitive, we may assume
without loss of generality that $t=cu+v$ for some $c\in A$. Given any $r\in R\cap\j$ (resp. $s\in R\cap\j$), we consider the unitary transvection $\tau_{r,u}$ (resp. $\tau_{s,u}$)
and the unitary transformation $g$, represented by the matrix
$$G=\left(
\begin{array}{cc}
(1+r)^{-1} & 0 \\
0 & 1+r \\
\end{array}
\right).
$$
Then $\tau_{r,u}$ (resp. $\tau_{s,u}$) as well as $g$ belong to $\Omega(\j)$, and we have
\begin{equation}\label{p2}
\tau_{r,u}t-t=ru\quad (\text{resp. }\tau_{s,u}t-t=su)
\end{equation}
and
\begin{equation}\label{p3}
gt-t=rt+c ((1+r)^{-1} - (1+r))u.
\end{equation}
Since $\j=R\cap\j\oplus S\cap\j$, (\ref{p1})-(\ref{p3}) show that given any $w\in U^\perp$ there is $d\in S\cap\j$ (resp. $d\in R\cap\j$) such that
$$
w+U\equiv du+U\mod I.
$$
Now $|N\cap (1+\j)|=|S\cap\j|$ by Lemma \ref{sn} and $|N\cap (1+\i)|=|S\cap\i|$ by Lemma \ref{nis}, so
$$
I(2)\leq |S\cap\j/S\cap\i|=|N\cap (1+\j)/N\cap (1+\i)|.
$$
(In the ramified even case, for any ideal $k$ of $A$ we have a bijection $R\cap\k\to S\cap\k$, given by $r\mapsto \pi r$. Thus, in this case,
$$
I(2)\leq |R\cap\j/R\cap\i|=|S\cap\j/S\cap\i|=|N\cap (1+\j)/N\cap (1+\i)|.)
$$
\end{proof}

\begin{lemma}\label{m2noniso} Suppose $m=2$ and $h$ is hermitian and non isotropic. Then
$$
I(2)\leq |N\cap (1+\j)/N\cap (1+\i)|.
$$
\end{lemma}

\begin{proof} By \cite[Lemma 3.7]{CHQS}, $h(t,t)$ must be a unit. By \cite[Lemma 2.3]{CHQS} there is $v\in V$ such that
$\{t,v\}$ is an orthogonal basis of $V$. Thus, relative to this basis, $h$ has matrix
$$\left(
    \begin{array}{cc}
      e & 0 \\
      0 & d \\
    \end{array}
  \right),
$$
where $e,d\in R^*$. We claim that given any $a\in \j$ there is $g\in\Omega(\j)$ such that $gt-t=bt+av$ for some $b\in R\cap\i$.
We first look for $b\in R\cap\i$ such that $h((1+b)t+av,(1+b)t,av)=e$. This translates into $(1+b)(1+b)^*e+aa^* d=e$.
Since $\i^2=(0)$, this becomes $2be+aa^* d=0$, that is $b=-e^{-1}daa^*/2\in R\cap\i$.

Next let $a'=-a^*e^{-1}d\in\j$. Arguing much as above, we can solve $h(a' t+(1+b')v,a't+(1+b')v)=d$ for $b'\in R\cap\i$.

We may now define $g$ by means of the matrix
$$
G=\left(
                              \begin{array}{cc}
                                1+b & a' \\
                                a & 1+b' \\
                              \end{array}
                            \right).
$$
Since $\i\j=(0)$, we see that $g\in\U(V)$ and a fortiori $g\in\Omega(\j)$. It is also clear that $gt-t=bt+av$, as claimed.
This and (\ref{p1}) yield the desired inequality, much as in the end of the proof of Lemma \ref{m2}.
\end{proof}

\begin{prop}\label{index} We have $I(m)\leq |N\cap (1+\j)/N\cap (1+\i)|$.
\end{prop}

\begin{proof} This follows from Lemmas \ref{m3}, \ref{m2} and \ref{m2noniso}.
\end{proof}

\begin{note}{\rm The argument given above to compute the index of the image of $\Gamma$ works as well to show the following.
Fix $t\in V$ primitive and consider the function $\U(V)\to V$, given by $g\mapsto gt-t$. Let $E$ be the subgroup of $V$ generated by its image.
Then $[V:E]=|S|$, the number of skew hermitian elements of $A$.}
\end{note}

\section{Nonabelian Clifford theory}\label{prueba2}

\noindent{\it Proof of Theorem \ref{mainnonabelian}.} Since $\j^2\subseteq\i$, the congruence subgroup $\Omega(\j)$ acts trivially on $U^\perp/U$. Since $S'(0,u)=1_Z$
for all $u\in U$ and $f(U,U^\perp)=0$, we have
$$
W'(g|_{U^\perp})S'(k)W'(g|_{U^\perp})^{-1}=S'({}^g k)=S'(k),\quad g\in\Omega(\j), k\in H(U^\perp).
$$
But $S':H(U^\perp)\to\GL(Z)$ is irreducible, so by Schur's Lemma, $W'(g|_{U^\perp})$ is a scalar operator for every $g\in\Omega(\j)$. Now
$$
W_t(g)=\mu(h(gt,t))S'(0,gt-t)W'(g|_{U^\perp}),\quad g\in\Omega(\j),
$$
so the $\Omega(\j)$-invariant subspaces of $Z=Z_t$ via $W_t$ are precisely the subspaces of $Z$ invariant under all $S'(0,gt-t)$, where $g\in\Omega(\j)$. These subspaces
are exactly the same as those invariant under all $S'(k)$, where $k$ runs through the subgroup $H_t$ of $H(U^\perp)$.

Now $Z_t(\phi)$  is a $\Omega(\j)$-invariant subspace of $Z$ via $W_t$. By above,
$Z_t(\phi)$ is an $H_t$-invariant subspace of $Z$ via $S'$.

We claim that $Z_t(\phi)$ is an irreducible $H_t$-module via $S'$. By above, this implies
that $Z_t(\phi)$ is an irreducible $\Omega(\j)$-module via $W_t$.

By Proposition \ref{index}
\begin{equation}
\label{menor}
[H(U^\perp):H_t]=[U^\perp/U:\Gamma(\Omega(\j))]\leq |N\cap (1+\j)/N\cap (1+\i)|.
\end{equation}

Consider the $H(U^\perp)$-module $\ind_{H_t}^{H(U^\perp)} Z_t(\phi)$. We infer from (\ref{menor}) and Lemma \ref{dim} that
\begin{equation}
\label{menor2}
\dim(\ind_{H_t}^{H(U^\perp)} Z_t(\phi))=\dim Z_t(\phi)[H(U^\perp):H_t] \leq \dim Z.
\end{equation}

Now $\ind_{H_t}^{H(U^\perp)} Z_t(\phi)$ is a nonzero $H(U^\perp)$-module where the normal subgroup $(R,U)$ acts via $\lambda$
(this is true before inducing, and also after inducing because $f(U,U^\perp)=0$). There is a unique, up to isomorphism, irreducible $H(U^\perp)$-module where $(R,U)$ acts via $\lambda$, namely the Schr${\rm\ddot{o}}$dinger module of type $\lambda$, that is, $Z$.
Thus $\ind_{H_t}^{H(U^\perp)} Z_t(\phi)$ is isomorphic to the direct sum of various copies of $Z$. We deduce from (\ref{menor2}) that
$\ind_{H_t}^{H(U^\perp)} Z_t(\phi)\cong Z$.
But $Z$ is an irreducible $H(U^\perp)$-module, so $Z_t(\phi)$ is an irreducible $H_t$-module and the stabilizer of $Z_t(\phi)$ in $H(U^\perp)$ must be exactly $H_t$. Moreover, since equality prevails in (\ref{menor2}), we can use Lemma \ref{dim} to see that equality prevails in (\ref{menor}) as well.

Finally, since the action of $\Omega(\j)$ extends to $G_t N$, it is clear that $G_t N$ stabilizes $Z_t(\phi)$. Suppose that $g\in\U(V)$
stabilizes $Z_t(\phi)$. As $\Omega(\i)$ acts on $Z_t(\phi)$ via $\alpha_t$, it follows that $g$ stabilizes $\alpha_t$, so $g\in G_t N$ by
Proposition \ref{inertia}. Thus the stabilizer of $Z_t(\phi)$ in $\U(V)$ is $G_t N$. That $Top(\phi)\cong \ind_{G_t N}^{\U(V)} Z_t(\phi)$ was proven
in Theorem \ref{mainabelian}. \qed

\section{Dimension of $Top(\phi)$}\label{prueba3}

Let $\phi\in \N$ and let $t$ be the only element of $T\cap\P$ such that the linear character
$\alpha_t$ of $\Omega(\i)$ enters $Top(\phi)$. Set $\overline{A}=A/\j A$ and $\overline{V}=V/\j V$. Let
$\overline{h}:\overline{V}\times \overline{V}\to \overline{A}$ be the nondegenerate form that $h$ induces on $V$, with
associated unitary group $\U(\overline{V})$. Set $\overline{t}=t+\j V\in \overline{V}$ and let $S_{\overline{t}}$ be the
pointwise stabilizer of $\overline{t}$ in $\U(\overline{V})$. We write $e$ for the nilpotency degree of $\r$. By Theorem \ref{mainabelian}, we have
\begin{equation}
\label{dime0}
\dim Top(\phi)=\dim Z_t(\phi)[\U(V):NG_t].
\end{equation}
Here
$$
\dim Z_t(\phi)=\frac{\sqrt{[\j V:\i V]}}{|N\cap (1+\j)/N\cap (1+\i)|}
$$
by Lemma \ref{dim}. Moreover,
$$
[\U(V):NG_t]=[\U(V):G_t]/[NG_t:G_t],
$$
where
$$
[\U(V):G_t]=[\U(V)/\Omega(\j):\Omega(\j)G_t/\Omega(\j)]=[\U(\overline{V}):S_{\overline t}]$$ and
$$
[NG_t:G_t]=[N:N\cap G_t]=[N:N\cap (1+\j)].
$$
Thus (\ref{dime0}) gives the general formula
\begin{equation}
\label{dime}
\dim Top(\phi)=\sqrt{[\j V:\i V]} [N:N\cap (1+\i)]^{-1} [\U(\overline{V}):S_{\overline{t}}].
\end{equation}
Here $|\k V |=|\k|^m$ for any ideal $\k$ of $A$. Moreover, $|N\cap (1+\i)|=|S\cap\i|$, by Lemma \ref{nis}, and
$$
|N|=\begin{cases} [A^\times:R^\times]\text{ if }*\text{ is unramified},\\
2[A^\times:R^\times]\text{ otherwise},\end{cases}
$$
because the norm map $A^\times\to R^\times$, given by $a\mapsto aa^*$, has image $R^\times$ if $*$ is unramified
and $R^{\times 2}$ otherwise. Furthermore,  $[\U(\overline{V}):S_{\overline{t}}]$ is given in \cite{CHQS} in the unramified and ramified even cases, and in \cite{CS}
in the ramified odd and symplectic cases.

Suppose first $*$ is symplectic, ramified odd, or ramified even with $h$ non isotropic. Then $[\U(\overline{V}):S_{\overline{t}}]$ is shown to be independent  of $t$ in
\cite[Theorem 7.1]{CS} in the symplectic and ramified odd cases,
and in \cite[Proposition 6.2]{CHQS} in the ramified even when $h$ is non isotropic. By (\ref{dime}), $\dim Top(\phi)$ is independent of $\phi$ and simplifies to
\begin{equation}
\label{dime2}
\dim Top(\phi)=\frac{\dim Top}{|\N|}=\frac{\sqrt{|V|}-\sqrt{[\r V:\min V]}}{|\N|}.
\end{equation}
In the symplectic case, $|\N|=2$ and $m=2n$, so (\ref{dime2}) reduces to
$$
\dim Top(\phi)=\frac{q^{ne}-q^{n(e-2)}}{2},
$$
In the ramified odd case, $|\N|=2q^{\ell-1}$, $m=2n$ and $e=2\ell-1$, so (\ref{dime2}) reduces to
$$
\dim Top(\phi)=\frac{q^{n(2\ell-1)}-q^{n(2\ell-3)}}{2q^{\ell-1}}.
$$
In the ramified even when $h$ is non isotropic, $|\N|=2q^{\ell-1}(q-1)$, $m=2$ and $e=2\ell$, so (\ref{dime2}) reduces to
$$
\dim Top(\phi)=\frac{q^{\ell-1}(q+1)}{2}.
$$

Unlike the above three cases, if $*$ is unramified or $*$ is ramified even with $h$ isotropic, then there are two possibilities for  $[\U(\overline{V}):S_{\overline{t}}]$, and hence $\dim Top(\phi)$, and this depends only
on whether $h(t,t)$ is a unit or not.

Indeed, suppose next $*$ is unramified. If $h(t,t)\in\r$, then the index $[\U(\overline{V}):S_{\overline{t}}]$ can easily be derived from
\cite[Corollary 5.6 and Theorem 7.3]{CHQS} by means of \cite[Ch. 11]{G}. This and (\ref{dime}) yield
\begin{equation}
\label{dime3}
\dim Top(\phi)=(q^{(e-1)(m-1)}+(-1)^m q^{(e-2)(m-1)})(q^m+(-1)^{m+1})/(q+1).
\end{equation}
On the other hand, if $h(t,t)\in A^\times$, then $S_{\overline{t}}\cong\U_{m-1}(\overline{A})$, so
$[\U(\overline{V}):S_{\overline{t}}]$ can be computed directly from
\cite[Corollary 5.6]{CHQS} and \cite[Ch. 11]{G}. This and (\ref{dime}) give
\begin{equation}
\label{dime4}
\dim Top(\phi)=q^{(e-1)(m-1)}(q^m+(-1)^{m+1})/(q+1).
\end{equation}
The primitivity of $\lambda$ together with (\ref{formu2}) readily give that $h(t,t)\in\r$ if and only if $\alpha_t|_{N\cap (1+\min)}$ is trivial, that is, if and only if $\phi|_{N\cap (1+\min)}$
is trivial. Thus, (\ref{dime3}) and (\ref{dime4}) are in complete agreement with \cite[Theorem 9.3]{S}.

Suppose finally that $*$ is ramified even with $h$ isotropic. The exact computations of $[\U(\overline{V}):S_{\overline{t}}]$, when $\i=\r^\ell$, and $\dim Top(\phi)$
are explicitly carried out in \cite[\S 8]{CHQS} (whether $h$ is isotropic or not).
We see that indeed $[\U(\overline{V}):S_{\overline{t}}]$ has two values, depending on whether $h(t,t)$ is a unit or not (the latter case is impossible when $h$ is non isotropic). As indicated in \cite[\S 8]{CHQS}, $[\U(\overline{V}):S_{\overline{t}}]$ also depends on whether $\ell$ and~$m$ are even or odd, as well as on the type of form $h$ is chosen to be (of the two available types). But we fixed these parameters at the outset, so $[\U(\overline{V}):S_{\overline{t}}]$ only depends on whether $h(t,t)$ is a unit or not.


\end{document}